\documentclass[a4paper, 12pt]{article}
\usepackage[top=30truemm, bottom=30truemm, left=20truemm, right=20truemm]{geometry}
\usepackage{amsthm}
\usepackage{amsmath}
\usepackage{amssymb}
\usepackage{amsfonts}
\usepackage{mathtools}
\usepackage{mathrsfs}
\usepackage{bm}
\usepackage[hidelinks]{hyperref}
\newtheoremstyle{mystyle}
	{\baselineskip}
	{\baselineskip}
	{\itshape}
	{}
	{\bfseries}
	{.}
	{1em}
	{}

\theoremstyle{mystyle}
	\newtheorem{lem}{Lemma}[section]
	
	\newtheorem{prop}[lem]{Proposition}
	\newtheorem{thm}[lem]{Theorem}

\newtheoremstyle{mystyle2}
	{\baselineskip}
	{\baselineskip}
	{\normalfont}
	{}
	{\bfseries}
	{.}
	{1em}
	{}

\theoremstyle{mystyle2}
	\newtheorem{rem}[lem]{Remark}
	\newtheorem{definition}[lem]{Definition}
\newcommand{\R}{\mathbb{R}}
\newcommand{\Z}{\mathbb{Z}}
\renewcommand{\cal}[1]{\mathcal{#1}}
\newcommand{\scr}[1]{\mathscr{#1}}
\newcommand{\abs}[1]{\left\lvert #1 \right\rvert}
\newcommand{\norm}[1]{\left\lVert #1 \right\rVert}
\numberwithin{equation}{section}

\allowdisplaybreaks[4]
\begin{document}

\title{On the optimal decay rate of global solutions to the convection-diffusion equation with zero-mass initial data}

\author{Yi C. Huang and Ryunosuke Kusaba\footnote{Corresponding author}}

\date{}

\maketitle



\begin{abstract}
	We consider the optimal decay rate of global solutions to the convection-diffusion equation with zero-mass initial data.
	We establish an improved decay estimate of the global solutions under the zero-mass condition.
	Moreover, we derive a self-similar asymptotic profile of the global solutions.
	This result shows that the improved decay rate is optimal and provides a necessary and sufficient condition for attaining this decay rate.
\end{abstract}
\section{Introduction}

In this paper, we consider the optimal decay rate of global solutions to the Cauchy problem for the convection-diffusion equation:
\begin{align*}
	\tag{P} \label{P}
	\begin{cases}
		\partial_{t} u- \Delta u=a \cdot \nabla f \left( u \right), &\qquad \left( t, x \right) \in \left( 0, \infty \right) \times \R^{n}, \\
		u \left( 0 \right) =u_{0}, &\qquad x \in \R^{n},
	\end{cases}
\end{align*}
where $u \colon \left[ 0, \infty \right) \times \R^{n} \to \R$ is an unknown scalar function, $u_{0} \colon \R^{n} \to \R$ is a given datum at $t=0$, $a \in \R^{n} \setminus \left\{ 0 \right\}$ is a given constant vector, and $f \colon \R \to \R$ is a (positively) homogeneous function of degree $p>1$ given by one of the following forms:
\begin{align*}
	f \left( \xi \right) = \xi^{p}, {\ } \abs{\xi}^{p}, {\ } \abs{\xi}^{p-1} \xi, \qquad \xi \in \R.
\end{align*}
In the case where $f \left( \xi \right) = \xi^{p}$, we assume that $p$ is an integer.
The convection-diffusion equation
\begin{align}
	\label{eq:conv-diff}
	\partial_{t} u- \Delta u=a \cdot \nabla f \left( u \right)
\end{align}
becomes the so-called viscous Burgers equation when $n=1$ and $f \left( \xi \right) = \xi^{2}$, which arises as a model for fluid flow, nonlinear acoustic wave propagation, and traffic flow (see \cite[Section 3.1]{Whitham}, \cite[Sections 3.7 and 4.5]{Hamilton-Blackstock}, and the references therein).
Hence, the convection-diffusion equation can be regarded as a generalization of the viscous Burgers equation with respect to the spatial dimension $n$ and the nonlinearity $f$.
In \eqref{eq:conv-diff}, the terms $\Delta u$ and $a \cdot \nabla f \left( u \right)$ describe linear diffusion and nonlinear convection in the direction of the constant vector $a$, respectively.

Before reviewing previous works on \eqref{P}, we introduce some notation.
For $q \in \left[ 1, \infty \right]$, let $L^{q} \left( \R^{n} \right)$ denote the standard Lebesgue space equipped with the norm $\norm{\,\cdot\,}_{L^{q}}$.
For $m>0$, the weighted $L^{1}$-space of order $m$ is defined by
\begin{align*}
	L^{1}_{m} \left( \R^{n} \right) \coloneqq \left\{ \varphi \in L^{1} \left( \R^{n} \right); {\,} \norm{\varphi}_{L^{1}_{m}} < \infty \right\},
\end{align*}
where
\begin{align*}
	\norm{\varphi}_{L^{1}_{m}} \coloneqq \int_{\R^{n}} \left( 1+ \abs{x}^{m} \right) \abs{\varphi \left( x \right)} dx.
\end{align*}
For a Banach space $E$ and an interval $I \subset \R$, let $BC \left( I; E \right)$ denote the set of all bounded continuous $E$-valued functions on $I$.
For Banach spaces $E$ and $F$ equipped with norms $\norm{\,\cdot\,}_{E}$ and $\norm{\,\cdot\,}_{F}$, respectively, we set
\begin{align*}
	\norm{\varphi}_{E \cap F} \coloneqq \norm{\varphi}_{E} + \norm{\varphi}_{F}, \qquad \varphi \in E \cap F.
\end{align*}

By virtue of the Duhamel principle, \eqref{P} can be written in the following integral form:
\begin{align}
	\label{eq:P_integral}
	u \left( t \right) =e^{t \Delta} u_{0} + \int_{0}^{t} a \cdot \nabla e^{\left( t- \tau \right) \Delta} f \left( u \left( \tau \right) \right) d \tau,
\end{align}
where $\left( e^{t \Delta}; t \geq 0 \right)$ is the heat semigroup defined by
\begin{align*}
	e^{t \Delta} \varphi \coloneqq \begin{cases}
		G_{t} \ast \varphi, &\qquad t>0, \\
		\varphi, &\qquad t=0,
	\end{cases}
\end{align*}
$G_{t} \colon \R^{n} \to \R$ is the heat kernel given by
\begin{align*}
	G_{t} \left( x \right) \coloneqq \left( 4 \pi t \right)^{- \frac{n}{2}} \exp \left( - \frac{\abs{x}^{2}}{4t} \right), \qquad x \in \R^{n},
\end{align*}
and $\ast$ is the convolution in $\R^{n}$.
Using this formulation, we define global solutions to \eqref{P} as follows.

\begin{definition}
	Let $u_{0} \in ( L^{1} \cap L^{\infty} ) \left( \R^{n} \right)$.
	Define
	\begin{align*}
		X \coloneqq BC \left( \left[ 0, \infty \right); L^{1} \left( \R^{n} \right) \right) \cap BC \left( \left( 0, \infty \right); L^{\infty} \left( \R^{n} \right) \right).
	\end{align*}
	A function $u \colon \left[ 0, \infty \right) \times \R^{n} \to \R$ is said to be a global solution to \eqref{P} if $u$ belongs to $X$ and satisfies \eqref{eq:P_integral} in $X$.
\end{definition}

The global well-posedness of \eqref{P} was established in \cite{Escobedo-Zuazua}, where a decay estimate of global solutions to \eqref{P} was also derived.
We refer the reader to \cite{Duro-Zuazua, Zuazua_2020} for improvements of these results.

\begin{prop}[{\cite[Proposition 1]{Escobedo-Zuazua}, \cite[Proposition 1]{Duro-Zuazua}, \cite[Theorems 4.1 and 4.2]{Zuazua_2020}}] \label{pro:P_global}
	Let $p>1$ and let $u_{0} \in ( L^{1} \cap L^{\infty} ) \left( \R^{n} \right)$.
	Then \eqref{P} has a unique global solution $u \in X$.
	Moreover, for any $q \in \left[ 1, \infty \right]$, there exists $C>0$ such that the estimate
	\begin{align}
		\label{eq:P_decay_gene}
		\norm{u \left( t \right)}_{L^{q}} \leq C \left( 1+t \right)^{- \frac{n}{2} \left( 1- \frac{1}{q} \right)}
	\end{align}
	holds for all $t>0$.
\end{prop}

We remark that the decay rate in \eqref{eq:P_decay_gene} coincides with that of the free solution $e^{t \Delta} u_{0}$.
For the decay estimate of the free solution, see Lemma \ref{lem:heat_decay} in Section \ref{sec:Heat}.
The above proposition implies that the effect of the linear diffusion is stronger than that of the nonlinear convection in the context of the global well-posedness.
In other words, this result does not reveal the effect of the nonlinear convection since the range of the exponent $p$ and the size of the initial datum $u_{0}$ are not restricted.
This fact leads us to investigate the asymptotic behavior of the global solution to \eqref{P} given in Proposition \ref{pro:P_global}.

In what follows, we frequently use two mappings $\cal{M}_{0} \colon L^{1} \left( \R^{n} \right) \to \R$ and $\cal{M}_{1} \colon L^{1}_{1} \left( \R^{n} \right) \to \R^{n}$ defined by
\begin{alignat*}{2}
	\cal{M}_{0} \left[ \varphi \right]
	&\coloneqq \int_{\R^{n}} \varphi \left( y \right) dy,
	&&\qquad \varphi \in L^{1} \left( \R^{n} \right), \\
	\cal{M}_{1} \left[ \varphi \right]
	&\coloneqq \int_{\R^{n}} y \varphi \left( y \right) dy,
	&&\qquad \varphi \in L^{1}_{1} \left( \R^{n} \right).
\end{alignat*}
The values $\cal{M}_{0} \left[ \varphi \right]$ and $\cal{M}_{1} \left[ \varphi \right]$ may be called moments of $\varphi$, and in particular, $\cal{M}_{0} \left[ \varphi \right]$ may be called the mass of $\varphi$.
We note that $\cal{M}_{1} \left[ \varphi \right]$ can be written as
\begin{align*}
	\cal{M}_{1} \left[ \varphi \right] = \left( \cal{M}_{e_{1}} \left[ \varphi \right], \ldots, \cal{M}_{e_{n}} \left[ \varphi \right] \right),
\end{align*}
where
\begin{align*}
	\cal{M}_{e_{j}} \left[ \varphi \right] \coloneqq \int_{\R^{n}} y_{j} \varphi \left( y \right) dy, \qquad j \in \left\{ 1, \ldots, n \right\}
\end{align*}
and $y_{j}$ denotes the $j$th component of $y \in \R^{n}$.
We will also consider the optimality of decay rates in estimates for the asymptotic behavior of the global solution to \eqref{P}.
A decay rate in an upper bound is said to be optimal if a lower bound with the same decay rate can be derived for some initial data.

In \cite{Escobedo-Zuazua, Escobedo-Vazquez-Zuazua_1, Escobedo-Vazquez-Zuazua_2, Zuazua_1994, Carpio_1996_CD}, the authors showed that the asymptotic behavior of the global solution $u$ to \eqref{P} with $f \left( \xi \right) = \abs{\xi}^{p-1} \xi$ changes dramatically depending on the exponent $p$ relative to the exponent $1+1/n$.
Under the same assumption as in Proposition \ref{pro:P_global} with $f \left( \xi \right) = \abs{\xi}^{p-1} \xi$, the classification of the asymptotic behavior is described as follows:
\begin{itemize}
	\item[(A)]
		If $p<1+1/n$, then for any $q \in \left[ 1, \infty \right)$, it holds that
		\begin{align*}
			\lim_{t \to \infty} t^{\frac{n+1}{2p} \left( 1- \frac{1}{q} \right)} \norm{u \left( t \right) - \cal{A}_{1, \mathrm{sub}} \left( t \right)}_{L^{q}} =0,
		\end{align*}
		where $\cal{A}_{1, \mathrm{sub}}$ is a unique self-similar solution to the following reduced equation:
		\begin{align*}
			\partial_{t} \cal{A}_{1, \mathrm{sub}} - \Delta_{\perp a} \cal{A}_{1, \mathrm{sub}} =a \cdot \nabla \left( \abs{\cal{A}_{1, \mathrm{sub}}}^{p-1} \cal{A}_{1, \mathrm{sub}} \right), \qquad \left( t, x \right) \in \left( 0, \infty \right) \times \R^{n}
		\end{align*}
		with $\cal{M}_{0} \left[ \cal{A}_{1, \mathrm{sub}} \left( t \right) \right] = \cal{M}_{0} \left[ u_{0} \right]$ for all $t>0$ (see \cite{Escobedo-Vazquez-Zuazua_1, Escobedo-Vazquez-Zuazua_2, Carpio_1996_CD}).
		Here, $\Delta_{\perp a}$ denotes the Laplacian on the hyperplane orthogonal to the constant vector $a$.
	\item[(B)]
		If $p=1+1/n$, then for any $q \in \left[ 1, \infty \right]$, it holds that
		\begin{align*}
			\lim_{t \to \infty} t^{\frac{n}{2} \left( 1- \frac{1}{q} \right)} \norm{u \left( t \right) - \cal{A}_{1, \mathrm{cri}} \left( t \right)}_{L^{q}} =0,
		\end{align*}
		where $\cal{A}_{1, \mathrm{cri}}$ is a unique self-similar solution to the convection-diffusion equation
		\begin{align*}
			\partial_{t} \cal{A}_{1, \mathrm{cri}} - \Delta \cal{A}_{1, \mathrm{cri}} =a \cdot \nabla \left( \abs{\cal{A}_{1, \mathrm{cri}}}^{p-1} \cal{A}_{1, \mathrm{cri}} \right), \qquad \left( t, x \right) \in \left( 0, \infty \right) \times \R^{n}
		\end{align*}
		with $\cal{M}_{0} \left[ \cal{A}_{1, \mathrm{cri}} \left( t \right) \right] = \cal{M}_{0} \left[ u_{0} \right]$ for all $t>0$ (see \cite{Escobedo-Zuazua, Zuazua_1994}).
	\item[(C)]
		If $p>1+1/n$, then for any $q \in \left[ 1, \infty \right]$, it holds that
		\begin{align*}
			\lim_{t \to \infty} t^{\frac{n}{2} \left( 1- \frac{1}{q} \right)} \norm{u \left( t \right) - \cal{A}_{1, \mathrm{sup}} \left( t \right)}_{L^{q}} =0,
		\end{align*}
		where $\cal{A}_{1, \mathrm{sup}}$ is a unique self-similar solution to the linear heat equation
		\begin{align*}
			\partial_{t} \cal{A}_{1, \mathrm{sup}} - \Delta \cal{A}_{1, \mathrm{sup}} =0, \qquad \left( t, x \right) \in \left( 0, \infty \right) \times \R^{n}
		\end{align*}
		with $\cal{M}_{0} \left[ \cal{A}_{1, \mathrm{sup}} \left( t \right) \right] = \cal{M}_{0} \left[ u_{0} \right]$ for all $t>0$ (see \cite{Escobedo-Zuazua}).
\end{itemize}

From the above results together with the self-similarity of the asymptotic profiles, we can determine the optimality of the decay rate in \eqref{eq:P_decay_gene}.
In fact, we see that the decay rate in \eqref{eq:P_decay_gene} is optimal when $p \geq 1+1/n$, whereas it is not optimal when $p<1+1/n$.
As for the case where $p<1+1/n$, it was shown in \cite{Escobedo-Vazquez-Zuazua_1, Escobedo-Vazquez-Zuazua_2, Carpio_1996_CD} that the global solution $u$ decays faster than the free solution $e^{t \Delta} u_{0}$.
More precisely, for any $q \in \left[ 1, \infty \right]$, there exists $C>0$ such that the estimate
\begin{align}
	\label{eq:P_decay_sub}
	\norm{u \left( t \right)}_{L^{q}} \leq Ct^{- \frac{n+1}{2p} \left( 1- \frac{1}{q} \right)}
\end{align}
holds for all $t>0$.
We note that
\begin{align*}
	\frac{n+1}{2p} > \frac{n}{2} \quad \Longleftrightarrow \quad p<1+ \frac{1}{n}.
\end{align*}
By using the asymptotic behavior of the global solution $u$ stated in (A) above, we can show that the decay rate in \eqref{eq:P_decay_sub} is optimal.
We emphasize that the optimal decay rates in \eqref{eq:P_decay_gene} and \eqref{eq:P_decay_sub} are attained if and only if $\cal{M}_{0} \left[ u_{0} \right] \neq 0$.

In the case where $p>1+1/n$, the asymptotic behavior of the global solution to \eqref{P} has been further investigated in \cite{Zuazua_1993, Fukuda-Sato, Kusaba, Yamamoto_2026} from the viewpoint of higher-order asymptotic expansions with optimal convergence rates.
In these papers, the initial datum is assumed to belong to the weighted $L^{1}$-spaces.
To clarify the role of the weight assumption on the initial datum, we review the result of \cite{Escobedo-Zuazua} concerning the first-order asymptotic expansion.
See also \cite{Duro-Zuazua, Zuazua_2020}.

\begin{prop}[{\cite[Theorem 2]{Escobedo-Zuazua}, \cite[Theorem 4]{Duro-Zuazua}, \cite[Theorem 5.3]{Zuazua_2020}}] \label{pro:P_1st-asymp_super}
	Let $p>1+1/n$ and let $u_{0} \in ( L^{1}_{1} \cap L^{\infty} ) \left( \R^{n} \right)$.
	Let $u \in X$ be the global solution to \eqref{P} given in Proposition \ref{pro:P_global}.
	Then for any $q \in \left[ 1, \infty \right]$, there exists $C>0$ such that the estimates
	\begin{align}
		\label{eq:P_1st-asymp_super}
		\norm{u \left( t \right) - \cal{A}_{1, \mathrm{sup}} \left( t \right)}_{L^{q}}
		\leq Ct^{- \frac{n}{2} \left( 1- \frac{1}{q} \right)} \times \begin{dcases}
			t^{- \sigma}, &\qquad p<1+ \frac{2}{n}, \\
			t^{- \frac{1}{2}} \log t, &\qquad p=1+ \frac{2}{n}, \\
			t^{- \frac{1}{2}}, &\qquad p>1+ \frac{2}{n}
		\end{dcases}
	\end{align}
	hold for all $t>2$, where
	\begin{align}
		\label{eq:sigma}
		\sigma
		&\coloneqq \frac{n}{2} \left( p-1 \right) - \frac{1}{2}, \\
		\label{eq:P_1st-asymp_profile}
		\cal{A}_{1, \mathrm{sup}} \left( t \right)
		&\coloneqq \cal{M}_{0} \left[ u_{0} \right] G_{t}.
	\end{align}
\end{prop}

Comparing Proposition \ref{pro:P_1st-asymp_super} with the statement (C) above, we observe that the weight assumption $\abs{x} u_{0} \in L^{1} \left( \R^{n} \right)$ improves the convergence rate in the first-order asymptotic expansion, namely, the decay rate of the remainder $u \left( t \right) - \cal{A}_{1, \mathrm{sup}} \left( t \right)$.
This situation is analogous to that for the free solution $e^{t \Delta} u_{0}$ (see Lemma \ref{lem:heat_asymp} in Section \ref{sec:Heat}).
Furthermore, Proposition \ref{pro:P_1st-asymp_super} suggests that the exponent $1+2/n$ may separate convergence rates in the first-order asymptotic expansion, reflecting the effect of the nonlinear convection.

To reveal the role of the new exponent $1+2/n$ in the first-order asymptotic expansion, the optimality of the decay rates in \eqref{eq:P_1st-asymp_super} should be determined.
For this purpose, the author of \cite{Zuazua_1993} established a second-order asymptotic expansion in each of the cases $p<1+2/n$, $p=1+2/n$, and $p>1+2/n$.
This result was improved in \cite{Kusaba} when $p<1+2/n$.
We also refer the reader to \cite{Yamamoto_2026} for the case $p=1+2/n$ in two dimensions.

\begin{prop}[{\cite[Theorem 1]{Zuazua_1993}, \cite[Theorem 2.1]{Kusaba}}] \label{pro:P_2nd-asymp_sub}
	Let $1+1/n<p<1+2/n$ and let $u_{0} \in ( L^{1}_{1} \cap L^{\infty} ) \left( \R^{n} \right)$.
	Let $u \in X$ be the global solution to \eqref{P} given in Proposition \ref{pro:P_global}.
	Then for any $q \in \left[ 1, \infty \right]$, there exists $C>0$ such that the estimates
	\begin{align}
		\label{eq:P_2nd-asymp_sub}
		\norm{u \left( t \right) - \cal{A}_{2, \mathrm{sub}} \left( t \right)}_{L^{q}}
		\leq Ct^{- \frac{n}{2} \left( 1- \frac{1}{q} \right)} \times \begin{dcases}
			t^{-2 \sigma}, &\qquad p<1+ \frac{3}{2n}, \\
			t^{- \frac{1}{2}} \log t, &\qquad p=1+ \frac{3}{2n}, \\
			t^{- \frac{1}{2}}, &\qquad p>1+ \frac{3}{2n}
		\end{dcases}
	\end{align}
	hold for all $t>2$, where
	\begin{align}
		\label{eq:P_2nd-asymp_sub_profile}
		\cal{A}_{2, \mathrm{sub}} \left( t \right) \coloneqq \cal{A}_{1, \mathrm{sup}} \left( t \right) + \int_{0}^{t} a \cdot \nabla e^{\left( t- \tau \right) \Delta} f \left( \cal{A}_{1, \mathrm{sup}} \left( \tau \right) \right) d \tau,
	\end{align}
	while $\sigma$ and $\cal{A}_{1, \mathrm{sup}} \left( t \right)$ are defined by \eqref{eq:sigma} and \eqref{eq:P_1st-asymp_profile}, respectively.
\end{prop}

\begin{prop}[{\cite[Theorem 2]{Zuazua_1993}, \cite[Proposition 1.5 and Appendix A]{Kusaba}, \cite[Theorem 1.1]{Yamamoto_2026}}] \label{pro:P_2nd-asymp_cri}
	Let $p=1+2/n$ and let $u_{0} \in ( L^{1}_{1} \cap L^{\infty} ) \left( \R^{n} \right)$.
	Let $u \in X$ be the global solution to \eqref{P} given in Proposition \ref{pro:P_global}.
	Then for any $q \in \left[ 1, \infty \right]$, there exists $C>0$ such that the estimate
	\begin{align}
		\label{eq:P_2nd-asymp_cri}
		\norm{u \left( t \right) - \cal{A}_{2, \mathrm{cri}} \left( t \right)}_{L^{q}}
		\leq Ct^{- \frac{n}{2} \left( 1- \frac{1}{q} \right) - \frac{1}{2}}
	\end{align}
	holds for all $t>2$, where
	\begin{align}
		\label{eq:P_2nd-asymp_cri_profile}
		\cal{A}_{2, \mathrm{cri}} \left( t \right) \coloneqq \cal{A}_{1, \mathrm{sup}} \left( t \right) +a \cdot \nabla G_{t} \int_{1}^{t} \cal{M}_{0} \left[ f \left( \cal{A}_{1, \mathrm{sup}} \left( \tau \right) \right) \right] d \tau,
	\end{align}
	and $\cal{A}_{1, \mathrm{sup}} \left( t \right)$ is defined by \eqref{eq:P_1st-asymp_profile}.
\end{prop}

\begin{prop}[{\cite[Theorem 3]{Zuazua_1993}, \cite[Proposition 1.6 and Appendix A]{Kusaba}}] \label{pro:P_2nd-asymp_super}
	Let $p>1+2/n$ and let $u_{0} \in ( L^{1}_{1} \cap L^{\infty} ) \left( \R^{n} \right)$.
	Let $u \in X$ be the global solution to \eqref{P} given in Proposition \ref{pro:P_global}.
	Then for any $q \in \left[ 1, \infty \right]$, it holds that
	\begin{align}
		\label{eq:P_2nd-asymp_sup}
		\lim_{t \to \infty} t^{\frac{n}{2} \left( 1- \frac{1}{q} \right) + \frac{1}{2}} \norm{u \left( t \right) - \cal{A}_{2, \mathrm{sup}} \left( t \right)}_{L^{q}} =0,
	\end{align}
	where
	\begin{align}
		\label{eq:P_2nd-asymp_sup_profile}
		\cal{A}_{2, \mathrm{sup}} \left( t \right)
		&\coloneqq \cal{A}_{1, \mathrm{sup}} \left( t \right) - \left( \cal{M}_{1} \left[ u_{0} \right] -a \cal{M}_{0} \left[ \psi_{0} \right] \right) \cdot \nabla G_{t}, \\
		\label{eq:psi_0}
		\psi_{0}
		&\coloneqq \int_{0}^{\infty} f \left( u \left( \tau \right) \right) d \tau,
	\end{align}
	and $\cal{A}_{1, \mathrm{sup}} \left( t \right)$ is defined by \eqref{eq:P_1st-asymp_profile}.
\end{prop}

The second-order asymptotic profiles $\cal{A}_{2, \mathrm{sub}} \left( t \right)$, $\cal{A}_{2, \mathrm{cri}} \left( t \right)$, and $\cal{A}_{2, \mathrm{sup}} \left( t \right)$ have the following self-similar structures (see, for example, \cite[Appendix B]{Kusaba}):
\begin{align*}
	\cal{A}_{2, \mathrm{sub}} \left( t \right) - \cal{A}_{1, \mathrm{sup}} \left( t \right) &=t^{- \sigma} \delta_{t} \cal{S}_{2, \mathrm{sub}}, \\
	\cal{A}_{2, \mathrm{cri}} \left( t \right) - \cal{A}_{1, \mathrm{sup}} \left( t \right) &=t^{- \frac{1}{2}} \left( \log t \right) \delta_{t} \cal{S}_{2, \mathrm{cri}}, \\
	\cal{A}_{2, \mathrm{sup}} \left( t \right) - \cal{A}_{1, \mathrm{sup}} \left( t \right) &=t^{- \frac{1}{2}} \delta_{t} \cal{S}_{2, \mathrm{sup}},
\end{align*}
where $\delta_{t}$ is the dilation acting on functions $\varphi \colon \R^{n} \to \R$ as
\begin{align}
	\label{eq:dilation}
	\left( \delta_{t} \varphi \right) \left( x \right) =t^{- \frac{n}{2}} \varphi \left( t^{- \frac{1}{2}} x \right), \qquad x \in \R^{n},
\end{align}
while $\cal{S}_{2, \mathrm{sub}}$, $\cal{S}_{2, \mathrm{cri}}$, and $\cal{S}_{2, \mathrm{sup}}$ are functions on $\R^{n}$ defined by
\begin{align}
	\label{eq:P_2nd-asymp_sub_shape}
	\cal{S}_{2, \mathrm{sub}}
	&\coloneqq \int_{0}^{1} a \cdot \nabla e^{\left( 1- \theta \right) \Delta} f \left( \cal{A}_{1, \mathrm{sup}} \left( \theta \right) \right) d \theta, \\
	\label{eq:P_2nd-asymp_cri_shape}
	\cal{S}_{2, \mathrm{cri}}
	&\coloneqq \frac{1}{4 \pi} \left( 1+ \frac{2}{n} \right)^{- \frac{n}{2}} f \left( \cal{M}_{0} \left[ u_{0} \right] \right) a \cdot \nabla G_{1}, \\
	\label{eq:P_2nd-asymp_sup_shape}
	\cal{S}_{2, \mathrm{sup}}
	&\coloneqq - \left( \cal{M}_{1} \left[ u_{0} \right] -a \cal{M}_{0} \left[ \psi_{0} \right] \right) \cdot \nabla G_{1}.
\end{align}
We note that by using the dilation $\delta_{t}$, the self-similarity of the first-order asymptotic profile $\cal{A}_{1, \mathrm{sup}} \left( t \right)$ is described as
\begin{align*}
	\cal{A}_{1, \mathrm{sup}} \left( t \right) = \delta_{t} \cal{A}_{1, \mathrm{sup}} \left( 1 \right).
\end{align*}
By combining Propositions \ref{pro:P_2nd-asymp_sub}, \ref{pro:P_2nd-asymp_cri}, and \ref{pro:P_2nd-asymp_super} with the self-similar structures of the second-order asymptotic profiles, the optimality of the decay rates in \eqref{eq:P_1st-asymp_super} can be determined.

\begin{prop}[{\cite[Propositions B.1, B.4, and B.6]{Kusaba}}] \label{pro:P_1st-asymp_optimal}
	Under the same assumption as in Proposition \ref{pro:P_1st-asymp_super}, the following assertions hold:
	\begin{itemize}
		\item[\rm (1)]
			If $p<1+2/n$, then for any $q \in \left[ 1, \infty \right]$, it holds that
			\begin{align*}
				\lim_{t \to \infty} t^{\frac{n}{2} \left( 1- \frac{1}{q} \right) + \sigma} \norm{u \left( t \right) - \cal{A}_{1, \mathrm{sup}} \left( t \right)}_{L^{q}} = \norm{\cal{S}_{2, \mathrm{sub}}}_{L^{q}},
			\end{align*}
			where $\sigma$, $\cal{A}_{1, \mathrm{sup}} \left( t \right)$, and $\cal{S}_{2, \mathrm{sub}}$ are defined by \eqref{eq:sigma}, \eqref{eq:P_1st-asymp_profile}, and \eqref{eq:P_2nd-asymp_sub_shape}, respectively.
			Moreover, $\norm{\cal{S}_{2, \mathrm{sub}}}_{L^{q}} >0$ if and only if $\cal{M}_{0} \left[ u_{0} \right] \neq 0$.
		\item[\rm (2)]
			If $p=1+2/n$, then for any $q \in \left[ 1, \infty \right]$, it holds that
			\begin{align*}
				\lim_{t \to \infty} t^{\frac{n}{2} \left( 1- \frac{1}{q} \right) + \frac{1}{2}} \left( \log t \right)^{-1} \norm{u \left( t \right) - \cal{A}_{1, \mathrm{sup}} \left( t \right)}_{L^{q}} = \norm{\cal{S}_{2, \mathrm{cri}}}_{L^{q}},
			\end{align*}
			where $\cal{A}_{1, \mathrm{sup}} \left( t \right)$ and $\cal{S}_{2, \mathrm{cri}}$ are defined by \eqref{eq:P_1st-asymp_profile} and \eqref{eq:P_2nd-asymp_cri_shape}, respectively.
			Moreover, \\
			$\norm{\cal{S}_{2, \mathrm{cri}}}_{L^{q}} >0$ if and only if $\cal{M}_{0} \left[ u_{0} \right] \neq 0$.
		\item[\rm (3)]
			If $p>1+2/n$, then for any $q \in \left[ 1, \infty \right]$, it holds that
			\begin{align*}
				\lim_{t \to \infty} t^{\frac{n}{2} \left( 1- \frac{1}{q} \right) + \frac{1}{2}} \norm{u \left( t \right) - \cal{A}_{1, \mathrm{sup}} \left( t \right)}_{L^{q}} = \norm{\cal{S}_{2, \mathrm{sup}}}_{L^{q}},
			\end{align*}
			where $\cal{A}_{1, \mathrm{sup}} \left( t \right)$ and $\cal{S}_{2, \mathrm{sup}}$ are defined by \eqref{eq:P_1st-asymp_profile} and \eqref{eq:P_2nd-asymp_sup_shape}, respectively.
			Moreover, \\
			$\norm{\cal{S}_{2, \mathrm{sup}}}_{L^{q}} >0$ if and only if $\cal{M}_{1} \left[ u_{0} \right] -a \cal{M}_{0} \left[ \psi_{0} \right] \neq 0$, where $\psi_{0}$ is defined by \eqref{eq:psi_0}.
	\end{itemize}
\end{prop}

The above proposition implies that the decay rates in \eqref{eq:P_1st-asymp_super} are all optimal and provides necessary and sufficient conditions for attaining these decay rates.
From this result, we conclude that the exponent $1+2/n$ is critical in the first-order asymptotic expansion.

Following the above discussion on the first-order asymptotic expansion, the authors of \cite{Fukuda-Sato, Kusaba} investigated the optimality of the decay rates in \eqref{eq:P_2nd-asymp_sub} and \eqref{eq:P_2nd-asymp_cri} via third-order asymptotic expansions with self-similar structures.
We note that the decay rate in \eqref{eq:P_2nd-asymp_sup} seems optimal since it coincides with that in the second-order asymptotic expansion of the free solution $e^{t \Delta} u_{0}$ (see Lemma \ref{lem:heat_asymp} (3) in Section \ref{sec:Heat}).
In \cite{Fukuda-Sato, Kusaba}, the following propositions were established.

\begin{prop}[{\cite[Theorems 2.4, 2.5, and 2.6]{Kusaba}}] \label{pro:P_2nd-asymp_sub_optimal}
	Under the same assumption as in Proposition \ref{pro:P_2nd-asymp_sub}, the following assertions hold:
	\begin{itemize}
		\item[\rm (1)]
			If $p<1+3/ ( 2n )$, then for any $q \in \left[ 1, \infty \right]$, it holds that
			\begin{align*}
				\lim_{t \to \infty} t^{\frac{n}{2} \left( 1- \frac{1}{q} \right) +2 \sigma} \norm{u \left( t \right) - \cal{A}_{2, \mathrm{sub}} \left( t \right)}_{L^{q}} = \norm{\cal{S}_{3, \mathrm{sub}}}_{L^{q}},
			\end{align*}
			where
			\begin{align*}
				\cal{S}_{3, \mathrm{sub}} \coloneqq \int_{0}^{1} a \cdot \nabla e^{\left( 1- \theta \right) \Delta} \left[ \left( \cal{A}_{2, \mathrm{sub}} \left( \theta \right) - \cal{A}_{1, \mathrm{sup}} \left( \theta \right) \right) f' \left( \cal{A}_{1, \mathrm{sup}} \left( \theta \right) \right) \right] d \theta,
			\end{align*}
			while $\sigma$, $\cal{A}_{1, \mathrm{sup}} \left( t \right)$, and $\cal{A}_{2, \mathrm{sub}} \left( t \right)$ are defined by \eqref{eq:sigma}, \eqref{eq:P_1st-asymp_profile}, and \eqref{eq:P_2nd-asymp_sub_profile}, respectively.
			Moreover, $\norm{\cal{S}_{3, \mathrm{sub}}}_{L^{q}} >0$ if and only if $\cal{M}_{0} \left[ u_{0} \right] \neq 0$.
		\item[\rm (2)]
			If $p=1+3/ ( 2n )$, then for any $q \in \left[ 1, \infty \right]$, there exists $C>0$ such that the estimate
			\begin{align}
				\label{eq:P_2nd-asymp_cri_imp}
				\norm{u \left( t \right) - \cal{A}_{2, \mathrm{sub}} \left( t \right)}_{L^{q}} \leq Ct^{- \frac{n}{2} \left( 1- \frac{1}{q} \right) - \frac{1}{2}}
			\end{align}
			holds for all $t>4$, where $\cal{A}_{2, \mathrm{sub}} \left( t \right)$ is defined by \eqref{eq:P_2nd-asymp_sub_profile}.
		\item[\rm (3)]
			If $p>1+3/ ( 2n )$, then for any $q \in \left[ 1, \infty \right]$, it holds that
			\begin{align*}
				\lim_{t \to \infty} t^{\frac{n}{2} \left( 1- \frac{1}{q} \right) + \frac{1}{2}} \norm{u \left( t \right) - \cal{A}_{2, \mathrm{sub}} \left( t \right)}_{L^{q}} = \norm{\cal{S}_{3, \mathrm{sup}}}_{L^{q}},
			\end{align*}
			where
			\begin{align*}
				\cal{S}_{3, \mathrm{sup}} &\coloneqq - \left( \cal{M}_{1} \left[ u_{0} \right] -a \cal{M}_{0} \left[ \psi_{1} \right] \right) \cdot \nabla G_{1}, \\
				\psi_{1} &\coloneqq \int_{0}^{\infty} \left( f \left( u \left( \tau \right) \right) -f \left( \cal{A}_{1, \mathrm{sup}} \left( \tau \right) \right) \right) d \tau,
			\end{align*}
			while $\cal{A}_{1, \mathrm{sup}} \left( t \right)$ and $\cal{A}_{2, \mathrm{sub}} \left( t \right)$ are defined by \eqref{eq:P_1st-asymp_profile} and \eqref{eq:P_2nd-asymp_sub_profile}, respectively.
			Moreover, $\norm{\cal{S}_{3, \mathrm{sup}}}_{L^{q}} >0$ if and only if $\cal{M}_{1} \left[ u_{0} \right] -a \cal{M}_{0} \left[ \psi_{1} \right] \neq 0$.
	\end{itemize}
\end{prop}

\begin{prop}[{\cite[Theorem 2.1 and Remark 2.4]{Fukuda-Sato}}] \label{pro:P_2nd-asymp_cri_optimal}
	Under the same assumption as in Proposition \ref{pro:P_2nd-asymp_cri}, it holds that
	\begin{align*}
		\lim_{t \to \infty} t^{\frac{n}{2} \left( 1- \frac{1}{q} \right) + \frac{1}{2}} \norm{u \left( t \right) - \cal{A}_{2, \mathrm{cri}} \left( t \right)}_{L^{q}} = \norm{\cal{S}_{3, \mathrm{cri}}}_{L^{q}}
	\end{align*}
	for any $q \in \left[ 1, \infty \right]$, where
	\begin{align*}
		\cal{S}_{3, \mathrm{cri}}
		&\coloneqq - \left( \cal{M}_{1} \left[ u_{0} \right] -a \cal{M}_{0} \left[ \Psi_{1} \right] \right) \cdot \nabla G_{1} \\
		&\hspace{1cm} +f \left( \cal{M}_{0} \left[ u_{0} \right] \right) \int_{0}^{1} a \cdot \nabla e^{\left( 1- \theta \right) \Delta} \left[ G_{\theta}^{1+ \frac{2}{n}} - \frac{1}{4 \pi} \left( 1+ \frac{2}{n} \right)^{- \frac{n}{2}} \theta^{-1} G_{\theta} \right] d \theta, \\
		\Psi_{1}
		&\coloneqq \int_{0}^{1} f \left( u \left( \tau \right) \right) d \tau + \int_{1}^{\infty} \left( f \left( u \left( \tau \right) \right) -f \left( \cal{A}_{1, \mathrm{sup}} \left( \tau \right) \right) \right) d \tau,
	\end{align*}
	while $\cal{A}_{1, \mathrm{sup}} \left( t \right)$ and $\cal{A}_{2, \mathrm{cri}} \left( t \right)$ are defined by \eqref{eq:P_1st-asymp_profile} and \eqref{eq:P_2nd-asymp_cri_profile}, respectively.
\end{prop}

The functions $\cal{S}_{3, \mathrm{sub}}$, $\cal{S}_{3, \mathrm{sup}}$, and $\cal{S}_{3, \mathrm{cri}}$ describe the shapes of the new leading terms in the third-order asymptotic profiles.
A sufficient condition for ensuring $\norm{\cal{S}_{3, \mathrm{cri}}}_{L^{q}} >0$ is that $\cal{M}_{0} \left[ u_{0} \right] =0$ and $\cal{M}_{1} \left[ u_{0} \right] -a \cal{M}_{0} \left[ \Psi_{1} \right] \neq 0$; however, this condition is not necessary.
From Propositions \ref{pro:P_2nd-asymp_sub_optimal} and \ref{pro:P_2nd-asymp_cri_optimal}, we observe that the decay rates in \eqref{eq:P_2nd-asymp_sub} and \eqref{eq:P_2nd-asymp_cri} are optimal except for the case where $p=1+3/ ( 2n )$.
Although Proposition \ref{pro:P_2nd-asymp_sub_optimal} shows that the exponent $1+3/ ( 2n )$ is critical in the second-order asymptotic expansion, the optimality of the decay rate in \eqref{eq:P_2nd-asymp_cri_imp} is still an open problem.
As a byproduct of Theorem \ref{thm:P_decay_below} in the next section, we provide an answer to this problem.
For further results on higher-order asymptotic expansions, we refer the reader to \cite{Ishige-Kawakami_2013, Kusaba, Yamamoto_2026}.

The purpose of this paper is to investigate the case where $\cal{M}_{0} \left[ u_{0} \right] =0$.
In this case, the asymptotic profiles $\cal{A}_{1, \mathrm{sub}} \left( t \right)$, $\cal{A}_{1, \mathrm{cri}} \left( t \right)$, $\cal{A}_{1, \mathrm{sup}} \left( t \right)$, $\cal{A}_{2, \mathrm{sub}} \left( t \right)$, and $\cal{A}_{2, \mathrm{cri}} \left( t \right)$ vanish identically.
As a consequence, the above results on the asymptotic behavior of the global solution to \eqref{P} yield improvements of the decay rates in \eqref{eq:P_decay_gene} and \eqref{eq:P_decay_sub}.
Therefore, we are naturally led to identify the optimal decay rate of the global solution under the zero-mass condition $\cal{M}_{0} \left[ u_{0} \right] =0$.
This observation may be important for revealing the role of the quadratic nonlinear convection in the incompressible Navier--Stokes equations from the perspective of the convection-diffusion equation.
The Cauchy problem for the incompressible Navier--Stokes equations is given by
\begin{align}
	\label{NS}
	\begin{cases}
		\partial_{t} u- \Delta u+ \left( u \cdot \nabla \right) u+ \nabla p=0, &\qquad \left( t, x \right) \in \left( 0, \infty \right) \times \R^{n}, \\
		\nabla \cdot u=0, &\qquad \left( t, x \right) \in \left( 0, \infty \right) \times \R^{n}, \\
		u \left( 0 \right) =u_{0}, &\qquad x \in \R^{n},
	\end{cases}
\end{align}
where $u \colon \left[ 0, \infty \right) \times \R^{n} \to \R^{n}$ is an unknown vector function, $p \colon \left[ 0, \infty \right) \times \R^{n} \to \R$ is an unknown scalar function, and $u_{0} \colon \R^{n} \to \R^{n}$ is a given datum at $t=0$.
The unknown functions $u$ and $p$ represent the velocity field and the pressure of the fluid, respectively, and the term $\left( u \cdot \nabla \right) u$ describes nonlinear convection.
The divergence-free condition $\nabla \cdot u=0$ characterizes the incompressibility of the fluid and imposes the compatibility condition $\nabla \cdot u_{0} =0$ on the initial datum.
It is known that the condition $\nabla \cdot u_{0} =0$ implies $\cal{M}_{0} \left[ u_{0} \right] =0$ if $u_{0} \in L^{1} \left( \R^{n} \right)$ (see \cite[Theorem 2.5 (i)]{Miyakawa}).
This fact leads to a crucial difference in the asymptotic behavior of global solutions between the incompressible Navier--Stokes equations and semilinear parabolic equations such as the convection-diffusion equation.
Asymptotic expansions of global solutions to \eqref{NS} were obtained in \cite{Carpio_1996_NS, Fujigaki-Miyakawa, Yamamoto_2024, Yamamoto_2025}.
As an application, the optimal decay rate of the global solutions was identified in \cite{Miyakawa-Schonbek}.
See also \cite{Schonbek, Choe-Jin} for the optimal decay rate.
In this paper, we establish corresponding results for the convection-diffusion equation.

The rest of this paper is organized as follows.
In the next section, we present our main results: Theorems \ref{thm:P_decay_above}, \ref{thm:P_asymp}, and \ref{thm:P_decay_below}.
Section \ref{sec:Heat} is devoted to a basic estimate and asymptotic expansions of the free solution.
In Section \ref{sec:thm_1}, we provide the proof of Theorem \ref{thm:P_decay_above}.
In Section \ref{sec:thm_2-3}, we prove Theorems \ref{thm:P_asymp} and \ref{thm:P_decay_below}.

\section{Main results} \label{sec:main}

The first theorem provides an upper bound for global solutions to \eqref{P} under the zero-mass condition.

\begin{thm} \label{thm:P_decay_above}
	Let $p>1+1/ ( n+1 )$ and let $u_{0} \in ( L^{1}_{1} \cap L^{\infty} ) \left( \R^{n} \right)$ satisfy $\cal{M}_{0} \left[ u_{0} \right] =0$.
	Let $u \in X$ be the global solution to \eqref{P} given in Proposition \ref{pro:P_global}.
	Then the following assertions hold:
	\begin{itemize}
		\item[\rm (1)]
			If $p>1+1/n$, then for any $q \in \left[ 1, \infty \right]$, there exists $C_{0} >0$ such that the estimate
			\begin{align}
				\label{eq:P_decay_above_1}
				\norm{u \left( t \right)}_{L^{q}} \leq C_{0} \left( 1+t \right)^{- \frac{n}{2} \left( 1- \frac{1}{q} \right) - \frac{1}{2}}
			\end{align}
			holds for all $t>0$.
		\item[\rm (2)]
			If $p \leq 1+1/n$, then there exist $\varepsilon_{0} >0$ and $C_{1} >0$, independent of $\norm{u_{0}}_{L^{1}_{1} \cap L^{\infty}}$, such that the estimate
			\begin{align}
				\label{eq:P_decay_above_2}
				\norm{u \left( t \right)}_{L^{q}} \leq C_{1} \norm{u_{0}}_{L^{1}_{1} \cap L^{\infty}} \left( 1+t \right)^{- \frac{n}{2} \left( 1- \frac{1}{q} \right) - \frac{1}{2}}
			\end{align}
			holds for all $q \in \left[ 1, \infty \right]$ and $t>0$, provided that $\norm{u_{0}}_{L^{1}_{1} \cap L^{\infty}} < \varepsilon_{0}$.
	\end{itemize}
\end{thm}

\begin{rem} \label{rem:smallness}
	The assertion (2) in the above theorem remains valid even in the case where $p>1+1/n$, which implies that the dependence of the constant $C_{0}$ in \eqref{eq:P_decay_above_1} on the initial datum $u_{0}$ can be made explicit under the smallness assumption.
\end{rem}

We remark that the decay rate in \eqref{eq:P_decay_above_1} and \eqref{eq:P_decay_above_2} is faster than that in \eqref{eq:P_decay_sub} since
\begin{align*}
	\frac{n}{2} \left( 1- \frac{1}{q} \right) + \frac{1}{2} > \frac{n+1}{2p} \left( 1- \frac{1}{q} \right)
\end{align*}
for all $q \in \left[ 1, \infty \right]$.
In \cite{Escobedo-Zuazua}, a similar result was established in the case where $p=1+1/n$.
It was shown that if $u_{0} \in ( L^{1} \cap L^{\infty} ) \left( \R^{n} \right)$, $e^{\abs{x}^{2} /8} u_{0} \in L^{2} \left( \R^{n} \right)$, and $\cal{M}_{0} \left[ u_{0} \right] =0$, then for any $\varepsilon >0$ and $q \in \left[ 1, \infty \right]$, there exists $C>0$ such that the estimate
\begin{align*}
	\norm{u \left( t \right)}_{L^{q}} \leq Ct^{- \frac{n}{2} \left( 1- \frac{1}{q} \right) - \frac{1}{2} + \varepsilon}
\end{align*}
holds for all $t>1$.
Hence, Theorem \ref{thm:P_decay_above} improves the decay rate in the above estimate and extends the range of the exponent $p$, although the smallness of the initial datum $u_{0}$ is assumed when $p \leq 1+1/n$.

The second theorem presents a self-similar asymptotic profile of global solutions to \eqref{P} with the same decay rate as in \eqref{eq:P_decay_above_1} and \eqref{eq:P_decay_above_2}, under the zero-mass condition.

\begin{thm} \label{thm:P_asymp}
	Let $p>1+1/ ( n+1 )$ and let $u_{0} \in ( L^{1}_{1} \cap L^{\infty} ) \left( \R^{n} \right)$ satisfy $\cal{M}_{0} \left[ u_{0} \right] =0$.
	Let $u \in X$ be the global solution to \eqref{P} given in Proposition \ref{pro:P_global} and assume
	\begin{align}
		\label{eq:P_decay_assump}
		\sup_{q \in \left[ 1, \infty \right]} \sup_{t>0} \left( 1+t \right)^{\frac{n}{2} \left( 1- \frac{1}{q} \right) + \frac{1}{2}} \norm{u \left( t \right)}_{L^{q}} < \infty.
	\end{align}
	Then for any $q \in \left[ 1, \infty \right]$, it holds that
	\begin{align}
		\label{eq:P_asymp}
		\lim_{t \to \infty} t^{\frac{n}{2} \left( 1- \frac{1}{q} \right) + \frac{1}{2}} \norm{u \left( t \right) - \scr{A} \left( t \right)}_{L^{q}} =0,
	\end{align}
	where
	\begin{align*}
		\scr{A} \left( t \right) \coloneqq - \left( \cal{M}_{1} \left[ u_{0} \right] -a \cal{M}_{0} \left[ \psi_{0} \right] \right) \cdot \nabla G_{t},
	\end{align*}
	and $\psi_{0}$ is defined by \eqref{eq:psi_0}.
	In addition, the following assertions hold:
	\begin{itemize}
		\item[\rm (1)]
			The identity
			\begin{align}
				\label{eq:P_asymp_self-similar}
				\scr{A} \left( t \right) =t^{- \frac{1}{2}} \delta_{t} \scr{A} \left( 1 \right)
			\end{align}
			holds for all $t>0$, where $\delta_{t}$ is the dilation defined by \eqref{eq:dilation}.
		\item[\rm (2)]
			Let $q \in \left[ 1, \infty \right]$.
			If $\scr{B} \in C \left( \left( 0, \infty \right); L^{q} \left( \R^{n} \right) \right)$ satisfies $\scr{B} \left( t \right) =t^{- \frac{1}{2}} \delta_{t} \scr{B} \left( 1 \right)$ for any $t>0$ and
			\begin{align*}
				\lim_{t \to \infty} t^{\frac{n}{2} \left( 1- \frac{1}{q} \right) + \frac{1}{2}} \norm{u \left( t \right) - \scr{B} \left( t \right)}_{L^{q}} =0,
			\end{align*}
			then $\scr{A} \left( t \right) = \scr{B} \left( t \right)$ for all $t>0$.
	\end{itemize}
\end{thm}

In the above theorem, the assertion (1) describes the self-similarity of the asymptotic profile $\scr{A} \left( t \right)$, while the assertion (2) shows the uniqueness of the self-similar asymptotic profile.
Since $1+1/ ( n+1 ) <1+1/n$, Theorem \ref{thm:P_asymp} implies that the linear diffusion dominates the nonlinear convection under the zero-mass condition even in the case where $p \leq 1+1/n$.
We note that the asymptotic profile $\scr{A} \left( t \right)$ is equal to the second-order asymptotic profile $\cal{A}_{2, \mathrm{sup}} \left( t \right)$ with $\cal{M}_{0} \left[ u_{0} \right] =0$, where $\cal{A}_{2, \mathrm{sup}} \left( t \right)$ is defined by \eqref{eq:P_2nd-asymp_sup_profile}.

As an application of Theorem \ref{thm:P_asymp}, we show that the decay rate in \eqref{eq:P_decay_above_1} and \eqref{eq:P_decay_above_2} is optimal and provide a necessary and sufficient condition for attaining this decay rate.

\begin{thm} \label{thm:P_decay_below}
	Under the same assumption as in Theorem \ref{thm:P_asymp}, it holds that
	\begin{align}
		\label{eq:P_decay_below}
		\lim_{t \to \infty} t^{\frac{n}{2} \left( 1- \frac{1}{q} \right) + \frac{1}{2}} \norm{u \left( t \right)}_{L^{q}} = \norm{\scr{A} \left( 1 \right)}_{L^{q}}
	\end{align}
	for any $q \in \left[ 1, \infty \right]$, where
	\begin{align*}
		\scr{A} \left( 1 \right) =- \left( \cal{M}_{1} \left[ u_{0} \right] -a \cal{M}_{0} \left[ \psi_{0} \right] \right) \cdot \nabla G_{1},
	\end{align*}
	and $\psi_{0}$ is defined by \eqref{eq:psi_0}.
	Moreover, $\norm{\scr{A} \left( 1 \right)}_{L^{q}} >0$ if and only if
	\begin{align}
		\label{eq:P_decay_moment}
		\cal{M}_{1} \left[ u_{0} \right] -a \cal{M}_{0} \left[ \psi_{0} \right] \neq 0.
	\end{align}
\end{thm}

Under the same assumption as in Theorem \ref{thm:P_asymp}, the condition \eqref{eq:P_decay_moment} is fulfilled if the initial datum $u_{0}$ satisfies $\cal{M}_{1} \left[ u_{0} \right] \neq 0$ and its amplitude is sufficiently small.
For details, see Proposition \ref{pro:appendix} in Appendix \ref{sec:app_moment}.
Theorem \ref{thm:P_decay_below} also shows that the decay rate in \eqref{eq:P_2nd-asymp_cri_imp} is optimal since $\cal{A}_{2, \mathrm{sub}} \left( t \right) \equiv 0$ when $\cal{M}_{0} \left[ u_{0} \right] =0$, where $\cal{A}_{2, \mathrm{sub}} \left( t \right)$ is defined by \eqref{eq:P_2nd-asymp_sub_profile}.

In \cite{Karch-Schonbek}, the authors considered the same problem under different assumptions.
They assumed the following two conditions for a fixed parameter $\beta \in \left( 0, 1 \right)$, in place of the assumption $\abs{x} u_{0} \in L^{1} \left( \R^{n} \right)$:
\begin{itemize}
	\item[(i)]
		$\scr{F}^{-1} \abs{\xi}^{- \beta} \scr{F} u_{0} \in L^{1} \left( \R^{n} \right)$, where $\scr{F}$ and $\scr{F}^{-1}$ denote the Fourier transform and its inverse, respectively.
	\item[(ii)]
		There exists $C>0$ such that the estimate
		\begin{align*}
			\norm{e^{t \Delta} u_{0}}_{L^{1}} \leq Ct^{- \frac{\beta}{2}}
		\end{align*}
		holds for all $t>0$.
\end{itemize}
These conditions are fulfilled if $u_{0} \in L^{1}_{\beta} \left( \R^{n} \right)$ and $\cal{M}_{0} \left[ u_{0} \right] =0$ (see \cite[Remark 3.2]{Karch-Schonbek}).
They obtained results corresponding to Theorems \ref{thm:P_decay_above} and \ref{thm:P_asymp} in the case where $p>1+1/ ( n+ \beta )$.
Moreover, they showed that the exponent $1+1/ ( n+ \beta )$ is critical for the asymptotic behavior of global solutions to \eqref{P}.
From this point of view, our results settle the case $\beta =1$, which was mentioned as an open problem in \cite[Remark 3.6]{Karch-Schonbek}.
As in the case where $\beta \in \left( 0, 1 \right)$, the exponent $1+1/ ( n+1 )$ is expected to be critical for the asymptotic behavior; however, this problem remains open.
Concerning the case $\beta =1$, we refer the reader to \cite{Benachour-Karch-Laurencot_2004_CD}, in which a partial result for the one-dimensional case was obtained through the asymptotic analysis of global solutions to the viscous Hamilton--Jacobi equation developed in \cite{Benachour-Karch-Laurencot_2004_HJ}.

\begin{rem}
	It seems difficult to improve the decay rate in \eqref{eq:P_decay_above_1} and \eqref{eq:P_decay_above_2} even under additional assumptions on the moments of the initial datum $u_{0}$.
	As a next step, we may consider the case where $u_{0} \in ( L^{1}_{1} \cap L^{\infty} ) \left( \R^{n} \right)$ satisfies not only $\cal{M}_{0} \left[ u_{0} \right] =0$ but also $\cal{M}_{1} \left[ u_{0} \right] =0$.
	In this case, the free solution $e^{t \Delta} u_{0}$ satisfies
	\begin{align*}
		\lim_{t \to \infty} t^{\frac{n}{2} \left( 1- \frac{1}{q} \right) + \frac{1}{2}} \norm{e^{t \Delta} u_{0}}_{L^{q}} =0
	\end{align*}
	for any $q \in \left[ 1, \infty \right]$ (see Lemma \ref{lem:heat_asymp} (3) in Section \ref{sec:Heat}).
	Hence, we expect that a similar improvement holds for global solutions to \eqref{P}.
	Indeed, we see from Theorem \ref{thm:P_decay_below} that the decay rate in \eqref{eq:P_decay_above_1} and \eqref{eq:P_decay_above_2} is improved if $\cal{M}_{0} \left[ \psi_{0} \right] =0$.
	We now assume that the composite function $f \circ u$ is nonnegative.
	This assumption is fulfilled if the homogeneous function $f$ is nonnegative.
	For example, $f$ may be given by $f \left( \xi \right) = \abs{\xi}^{p}$ or by $f \left( \xi \right) = \xi^{p}$ with an even integer $p$.
	Under this assumption, the condition $\cal{M}_{0} \left[ \psi_{0} \right] =0$ implies $u \equiv 0$, which does not yield any meaningful result.
	Although there is room to investigate the case where $f \circ u$ changes sign, the main issue is to characterize the condition $\cal{M}_{0} \left[ \psi_{0} \right] =0$ in terms of the initial datum $u_{0}$.
\end{rem}

\begin{rem}
	Our main results may be extended to convection-diffusion equations with dispersion or variable diffusion.
	For the asymptotic behavior of global solutions to these equations, we refer the reader to \cite{Karch, Duro-Carpio, Fukuda, Fukuda-Irino} and the references therein.
\end{rem}

\section{Preliminaries} \label{sec:Heat}

We first recall the multi-index notation.
Let $\Z_{>0}$ denote the set of all positive integers and let $\Z_{\geq 0} \coloneqq \Z_{>0} \cup \left\{ 0 \right\}$.
For $\alpha = \left( \alpha_{1}, \ldots, \alpha_{n} \right) \in \Z_{\geq 0}^{n}$, we define
\begin{align*}
	\abs{\alpha} \coloneqq \sum_{j=1}^{n} \alpha_{j},
	\qquad
	\partial^{\alpha} = \partial_{x}^{\alpha} \coloneqq \prod_{j=1}^{n} \partial_{j}^{\alpha_{j}},
\end{align*}
where $\partial_{j} \coloneqq \partial / \partial x_{j}$.

We next summarize a basic estimate and asymptotic expansions of the free solution.

\begin{lem}[{cf. \cite[Section 1.1.3]{Giga-Giga-Saal}}] \label{lem:heat_decay}
	Let $1 \leq q \leq r \leq \infty$ and let $\alpha \in \Z_{\geq 0}^{n}$.
	Then there exists $C>0$ such that the estimate
	\begin{align*}
		\norm{\partial^{\alpha} e^{t \Delta} \varphi}_{L^{r}} \leq Ct^{- \frac{n}{2} \left( \frac{1}{q} - \frac{1}{r} \right) - \frac{\abs{\alpha}}{2}} \norm{\varphi}_{L^{q}}
	\end{align*}
	holds for all $\varphi \in L^{q} \left( \R^{n} \right)$ and $t>0$.
\end{lem}

\begin{lem}[{cf. \cite[Th\'{e}or\`{e}me 4]{Duoandikoetxea-Zuazua}}] \label{lem:heat_asymp}
	The following assertions hold:
	\begin{itemize}
		\item[\rm (1)]
			If $\varphi \in L^{1} \left( \R^{n} \right)$, then for any $q \in \left[ 1, \infty \right]$ and $\alpha \in \Z_{\geq 0}^{n}$, it holds that
			\begin{align*}
				\lim_{t \to \infty} t^{\frac{n}{2} \left( 1- \frac{1}{q} \right) + \frac{\abs{\alpha}}{2}} \norm{\partial^{\alpha} \left( e^{t \Delta} \varphi - \cal{M}_{0} \left[ \varphi \right] G_{t} \right)}_{L^{q}} =0.
			\end{align*}
		\item[\rm (2)]
			If $\varphi \in L^{1}_{1} \left( \R^{n} \right)$, then for any $q \in \left[ 1, \infty \right]$ and $\alpha \in \Z_{\geq 0}^{n}$, there exists $C>0$ such that the estimate
			\begin{align*}
				\norm{\partial^{\alpha} \left( e^{t \Delta} \varphi - \cal{M}_{0} \left[ \varphi \right] G_{t} \right)}_{L^{q}} \leq Ct^{- \frac{n}{2} \left( 1- \frac{1}{q} \right) - \frac{\abs{\alpha} +1}{2}} \norm{\abs{x} \varphi}_{L^{1}}
			\end{align*}
			holds for all $t>0$.
		\item[\rm (3)]
			If $\varphi \in L^{1}_{1} \left( \R^{n} \right)$, then for any $q \in \left[ 1, \infty \right]$ and $\alpha \in \Z_{\geq 0}^{n}$, it holds that
			\begin{align*}
				\lim_{t \to \infty} t^{\frac{n}{2} \left( 1- \frac{1}{q} \right) + \frac{\abs{\alpha} +1}{2}} \norm{\partial^{\alpha} \left( e^{t \Delta} \varphi - \cal{M}_{0} \left[ \varphi \right] G_{t} + \cal{M}_{1} \left[ \varphi \right] \cdot \nabla G_{t} \right)}_{L^{q}} =0.
			\end{align*}
	\end{itemize}
\end{lem}

\section{Proof of Theorem \ref{thm:P_decay_above}} \label{sec:thm_1}

We first give the proof of Theorem \ref{thm:P_decay_above} (1).
In this proof, the following two lemmas on higher-order asymptotic expansions of global solutions to \eqref{P} play an important role.

\begin{lem}[{\cite[Theorem 2.1]{Kusaba}}] \label{lem:P_asymp_higher_sub}
	Let $k \in \Z_{>0}$ and let $1+1/n<p<1+ ( k+1 ) / ( kn )$.
	Let $u_{0} \in ( L^{1}_{1} \cap L^{\infty} ) \left( \R^{n} \right)$ and let $u \in X$ be the global solution to \eqref{P} given in Proposition \ref{pro:P_global}.
	Then for any $q \in \left[ 1, \infty \right]$, there exists $C>0$ such that the estimates
	\begin{align*}
		\norm{u \left( t \right) - \cal{A}_{k+1, \mathrm{sub}} \left( t \right)}_{L^{q}} \leq Ct^{- \frac{n}{2} \left( 1- \frac{1}{q} \right)} \times \begin{dcases}
			t^{- \left( k+1 \right) \sigma}, &\qquad p<1+ \frac{k+2}{\left( k+1 \right) n}, \\
			t^{- \frac{1}{2}} \log t, &\qquad p=1+ \frac{k+2}{\left( k+1 \right) n}, \\
			t^{- \frac{1}{2}}, &\qquad p>1+ \frac{k+2}{\left( k+1 \right) n}
		\end{dcases}
	\end{align*}
	hold for all $t>2^{k}$, where
	\begin{align}
		\label{eq:P_higher-asymp_sub_profile}
		\cal{A}_{k+1, \mathrm{sub}} \left( t \right)
		&\coloneqq \cal{A}_{1, \mathrm{sup}} \left( t \right) + \int_{0}^{1} a \cdot \nabla e^{\left( t- \tau \right) \Delta} f \left( \cal{A}_{1, \mathrm{sup}} \left( \tau \right) \right) d \tau \nonumber \\
		&\hspace{1cm} + \begin{dcases}
			\int_{1}^{t} a \cdot \nabla e^{\left( t- \tau \right) \Delta} f \left( \cal{A}_{1, \mathrm{sup}} \left( \tau \right) \right) d \tau, &\qquad k=1, \\
			\int_{1}^{t} a \cdot \nabla e^{\left( t- \tau \right) \Delta} f \left( \cal{A}_{k, \mathrm{sub}} \left( \tau \right) \right) d \tau, &\qquad k \geq 2,
		\end{dcases}
	\end{align}
	while $\sigma$ and $\cal{A}_{1, \mathrm{sup}} \left( t \right)$ are defined by \eqref{eq:sigma} and \eqref{eq:P_1st-asymp_profile}, respectively.
\end{lem}

Lemma \ref{lem:P_asymp_higher_sub} with $k=1$ reduces to Proposition \ref{pro:P_2nd-asymp_sub}.
We remark that $\left( k+1 \right) \sigma <1/2$ if and only if $p<1+ ( k+2 ) / ( \left( k+1 \right) n )$ and that $\cal{A}_{k+1, \mathrm{sub}} \left( t \right)$ is defined recursively through an iteration based on $\cal{A}_{1, \mathrm{sup}} \left( t \right)$.

\begin{lem}[{\cite[Theorem 2.2]{Kusaba}}] \label{lem:P_asymp_higher_cri}
	Let $k \in \Z_{>0}$ and let $p=1+ ( k+2 ) / ( \left( k+1 \right) n )$.
	Let $u_{0} \in ( L^{1}_{1} \cap L^{\infty} ) \left( \R^{n} \right)$ and let $u \in X$ be the global solution to \eqref{P} given in Proposition \ref{pro:P_global}.
	Then for any $q \in \left[ 1, \infty \right]$, there exists $C>0$ such that the estimate
	\begin{align*}
		\norm{u \left( t \right) - \cal{A}_{k+2, \mathrm{cri}} \left( t \right)}_{L^{q}} \leq Ct^{- \frac{n}{2} \left( 1- \frac{1}{q} \right) - \frac{1}{2}}
	\end{align*}
	holds for all $t>2^{k+1}$, where
	\begin{align}
		\label{eq:P_higher-asymp_cri_profile}
		\cal{A}_{k+2, \mathrm{cri}} \left( t \right)
		&\coloneqq \cal{A}_{k+1, \mathrm{sub}} \left( t \right) \nonumber \\
		&\hspace{1cm} + \begin{dcases}
			a \cdot \nabla G_{t} \int_{1}^{t} \cal{M}_{0} \left[ f \left( \cal{A}_{2, \mathrm{sub}} \left( \tau \right) \right) -f \left( \cal{A}_{1, \mathrm{sup}} \left( \tau \right) \right) \right] d \tau, &\qquad k=1, \\
			a \cdot \nabla G_{t} \int_{1}^{t} \cal{M}_{0} \left[ f \left( \cal{A}_{k+1, \mathrm{sub}} \left( \tau \right) \right) -f \left( \cal{A}_{k, \mathrm{sub}} \left( \tau \right) \right) \right] d \tau, &\qquad k \geq 2,
		\end{dcases}
	\end{align}
	while $\cal{A}_{1, \mathrm{sup}} \left( t \right)$ and $\cal{A}_{k+1, \mathrm{sub}} \left( t \right)$ are defined by \eqref{eq:P_1st-asymp_profile} and \eqref{eq:P_higher-asymp_sub_profile}, respectively.
\end{lem}

We are now in a position to prove Theorem \ref{thm:P_decay_above} (1).

\begin{proof}[Proof of Theorem \ref{thm:P_decay_above} (1)]
	Let $q \in \left[ 1, \infty \right]$.
	We first consider the case where $p \geq 1+2/n$.
	The assumption $\cal{M}_{0} \left[ u_{0} \right] =0$ implies $\cal{A}_{1, \mathrm{sup}} \left( t \right) \equiv 0$ and $\cal{A}_{2, \mathrm{cri}} \left( t \right) \equiv 0$ for all $t>1$, where $\cal{A}_{1, \mathrm{sup}} \left( t \right)$ and $\cal{A}_{2, \mathrm{cri}} \left( t \right)$ are defined by \eqref{eq:P_1st-asymp_profile} and \eqref{eq:P_2nd-asymp_cri_profile}, respectively.
	It follows from Propositions \ref{pro:P_1st-asymp_super} and \ref{pro:P_2nd-asymp_cri} that
	\begin{align*}
		\norm{u \left( t \right)}_{L^{q}} \leq Ct^{- \frac{n}{2} \left( 1- \frac{1}{q} \right) - \frac{1}{2}} \leq C \left( 1+t \right)^{- \frac{n}{2} \left( 1- \frac{1}{q} \right) - \frac{1}{2}}
	\end{align*}
	for any $t>2$.
	If $t \in \left( 0, 2 \right]$, then by \eqref{eq:P_decay_gene}, we obtain
	\begin{align*}
		\norm{u \left( t \right)}_{L^{q}} \leq C \left( 1+t \right)^{- \frac{n}{2} \left( 1- \frac{1}{q} \right)} \leq C \left( 1+t \right)^{- \frac{n}{2} \left( 1- \frac{1}{q} \right) - \frac{1}{2}}.
	\end{align*}

	We next consider the case where $1+1/n<p<1+2/n$.
	In this case, there exists $k \in \Z_{>0}$ such that either $1+ ( k+2 ) / ( \left( k+1 \right) n ) <p<1+ ( k+1 ) / ( kn )$ or $p=1+ ( k+2 ) / ( \left( k+1 \right) n )$ holds.
	We note that $\cal{A}_{k+1, \mathrm{sub}} \left( t \right) \equiv 0$ and $\cal{A}_{k+2, \mathrm{cri}} \left( t \right) \equiv 0$ for all $t>1$, due to the assumption $\cal{M}_{0} \left[ u_{0} \right] =0$.
	Here, $\cal{A}_{k+1, \mathrm{sub}} \left( t \right)$ and $\cal{A}_{k+2, \mathrm{cri}} \left( t \right)$ are defined by \eqref{eq:P_higher-asymp_sub_profile} and \eqref{eq:P_higher-asymp_cri_profile}, respectively.
	Therefore, applying Lemmas \ref{lem:P_asymp_higher_sub} and \ref{lem:P_asymp_higher_cri} gives
	\begin{align*}
		\norm{u \left( t \right)}_{L^{q}} \leq Ct^{- \frac{n}{2} \left( 1- \frac{1}{q} \right) - \frac{1}{2}} \leq C \left( 1+t \right)^{- \frac{n}{2} \left( 1- \frac{1}{q} \right) - \frac{1}{2}}
	\end{align*}
	for any $t>2^{k+1}$.
	If $t \in \left( 0, 2^{k+1} \right]$, then \eqref{eq:P_decay_gene} yields the desired estimate in the same manner as in the case where $p \geq 1+2/n$.
\end{proof}

We next present the proof of Theorem \ref{thm:P_decay_above} (2).
Let $u \in X$ be the global solution to \eqref{P} given in Proposition \ref{pro:P_global}.
We define a function $Q \colon \left( 0, \infty \right) \to \R$ by
\begin{align}
	\label{eq:Q}
	Q \left( t \right) \coloneqq \sup_{0<s \leq t} \left( 1+s \right)^{\frac{1}{2}} \left( \norm{u \left( s \right)}_{L^{1}} + \left( 1+s \right)^{\frac{n}{2}} \norm{u \left( s \right)}_{L^{\infty}} \right), \qquad t>0.
\end{align}
Then we see that $Q$ is continuous.

The following lemma provides an estimate of the Duhamel term.

\begin{lem} \label{lem:P_Duhamel_decay}
	Let $p>1+1/ ( n+1 )$ and let $u_{0} \in ( L^{1} \cap L^{\infty} ) \left( \R^{n} \right)$.
	Let $u \in X$ be the global solution to \eqref{P} given in Proposition \ref{pro:P_global}.
	Then for any $q \in \left[ 1, \infty \right]$, there exists $C>0$, independent of $\norm{u_{0}}_{L^{1} \cap L^{\infty}}$, such that the estimate
	\begin{align*}
		\norm{\int_{0}^{t} a \cdot \nabla e^{\left( t- \tau \right) \Delta} f \left( u \left( \tau \right) \right) d \tau}_{L^{q}} \leq C \left( 1+t \right)^{- \frac{n}{2} \left( 1- \frac{1}{q} \right) - \frac{1}{2}} Q \left( t \right)^{p}
	\end{align*}
	holds for all $t>0$, where $Q \left( t \right)$ is defined by \eqref{eq:Q}.
\end{lem}

\begin{proof}
	Let $q \in \left[ 1, \infty \right]$.
	We first consider the case where $0<t \leq 1$.
	By Lemma \ref{lem:heat_decay}, we have
	\begin{align*}
		\norm{\int_{0}^{t} a \cdot \nabla e^{\left( t- \tau \right) \Delta} f \left( u \left( \tau \right) \right) d \tau}_{L^{q}}
		&\leq \int_{0}^{t} \norm{a \cdot \nabla e^{\left( t- \tau \right) \Delta} f \left( u \left( \tau \right) \right)}_{L^{q}} d \tau \\
		&\leq C \int_{0}^{t} \left( t- \tau \right)^{- \frac{1}{2}} \norm{u \left( \tau \right)}_{L^{pq}}^{p} d \tau \\
		&\leq CQ \left( t \right)^{p} \int_{0}^{t} \left( t- \tau \right)^{- \frac{1}{2}} \left( 1+ \tau \right)^{- \frac{n}{2} \left( 1- \frac{1}{q} \right) - \frac{n}{2} \left( p-1 \right) - \frac{p}{2}} d \tau \\
		&\leq CQ \left( t \right)^{p} \int_{0}^{t} \left( t- \tau \right)^{- \frac{1}{2}} d \tau \\
		&\leq CQ \left( t \right)^{p} t^{\frac{1}{2}} \\
		&\leq CQ \left( t \right)^{p}.
	\end{align*}
	We next consider the case where $t>1$.
	It follows from Lemma \ref{lem:heat_decay} that
	\begin{align*}
		\norm{\int_{0}^{t/2} a \cdot \nabla e^{\left( t- \tau \right) \Delta} f \left( u \left( \tau \right) \right) d \tau}_{L^{q}}
		&\leq \int_{0}^{t/2} \norm{a \cdot \nabla e^{\left( t- \tau \right) \Delta} f \left( u \left( \tau \right) \right)}_{L^{q}} d \tau \\
		&\leq C \int_{0}^{t/2} \left( t- \tau \right)^{- \frac{n}{2} \left( 1- \frac{1}{q} \right) - \frac{1}{2}} \norm{u \left( \tau \right)}_{L^{p}}^{p} d \tau \\
		&\leq Ct^{- \frac{n}{2} \left( 1- \frac{1}{q} \right) - \frac{1}{2}} Q \left( t \right)^{p} \int_{0}^{t/2} \left( 1+ \tau \right)^{- \frac{n}{2} \left( p-1 \right) - \frac{p}{2}} d \tau \\
		&\leq Ct^{- \frac{n}{2} \left( 1- \frac{1}{q} \right) - \frac{1}{2}} Q \left( t \right)^{p}, \\
		\norm{\int_{t/2}^{t} a \cdot \nabla e^{\left( t- \tau \right) \Delta} f \left( u \left( \tau \right) \right) d \tau}_{L^{q}}
		&\leq \int_{t/2}^{t} \norm{a \cdot \nabla e^{\left( t- \tau \right) \Delta} f \left( u \left( \tau \right) \right)}_{L^{q}} d \tau \\
		&\leq C \int_{t/2}^{t} \left( t- \tau \right)^{- \frac{1}{2}} \norm{u \left( \tau \right)}_{L^{pq}}^{p} d \tau \\
		&\leq C \int_{t/2}^{t} \left( t- \tau \right)^{- \frac{1}{2}} \left( 1+ \tau \right)^{- \frac{n}{2} \left( 1- \frac{1}{q} \right) - \frac{n}{2} \left( p-1 \right) - \frac{p}{2}} d \tau \\
		&\leq Ct^{- \frac{n}{2} \left( 1- \frac{1}{q} \right) - \frac{n}{2} \left( p-1 \right) - \frac{p}{2}} Q \left( t \right)^{p} \int_{t/2}^{t} \left( t- \tau \right)^{- \frac{1}{2}} d \tau \\
		&\leq Ct^{- \frac{n}{2} \left( 1- \frac{1}{q} \right) - \frac{n}{2} \left( p-1 \right) - \frac{p}{2} + \frac{1}{2}} Q \left( t \right)^{p} \\
		&\leq Ct^{- \frac{n}{2} \left( 1- \frac{1}{q} \right) - \frac{1}{2}} Q \left( t \right)^{p}.
	\end{align*}
	Here, we have used the fact that
	\begin{align*}
		- \frac{n}{2} \left( p-1 \right) - \frac{p}{2} <-1 \quad \Longleftrightarrow \quad p>1+ \frac{1}{n+1}.
	\end{align*}
	Combining these estimates yields the desired conclusion.
\end{proof}

We are ready to prove Theorem \ref{thm:P_decay_above} (2).

\begin{proof}[Proof of Theorem \ref{thm:P_decay_above} (2)]
	Let $t>0$.
	Since $\cal{M}_{0} \left[ u_{0} \right] =0$, it follows from Lemmas \ref{lem:heat_decay} and \ref{lem:heat_asymp} that
	\begin{align*}
		\norm{e^{t \Delta} u_{0}}_{L^{q}}
		&\leq C \norm{u_{0}}_{L^{q}} \leq C \norm{u_{0}}_{L^{1} \cap L^{\infty}}, \\
		\norm{e^{t \Delta} u_{0}}_{L^{q}}
		&= \norm{e^{t \Delta} u_{0} - \cal{M}_{0} \left[ u_{0} \right] G_{t}}_{L^{q}} \leq Ct^{- \frac{n}{2} \left( 1- \frac{1}{q} \right) - \frac{1}{2}} \norm{\abs{x} u_{0}}_{L^{1}}
	\end{align*}
	for any $q \in \left[ 1, \infty \right]$.
	Hence, there exists $C_{2} >0$, depending only on $n$, such that
	\begin{align*}
		\left( 1+t \right)^{\frac{1}{2}} \left( \norm{e^{t \Delta} u_{0}}_{L^{1}} + \left( 1+t \right)^{\frac{n}{2}} \norm{e^{t \Delta} u_{0}}_{L^{\infty}} \right) \leq C_{2} \norm{u_{0}}_{L^{1}_{1} \cap L^{\infty}}.
	\end{align*}
	Applying this estimate and Lemma \ref{lem:P_Duhamel_decay} to \eqref{eq:P_integral}, we obtain
	\begin{align*}
		Q \left( t \right) \leq C_{2} \norm{u_{0}}_{L^{1}_{1} \cap L^{\infty}} +C_{3} Q \left( t \right)^{p}
	\end{align*}
	for all $t>0$, where $Q \left( t \right)$ is defined by \eqref{eq:Q} and $C_{3} >0$ is a constant depending only on $n$, $p$, and $a$.
	In addition, we see from the proof of Lemma \ref{lem:P_Duhamel_decay} that
	\begin{align*}
		\lim_{t \searrow 0} \norm{\int_{0}^{t} a \cdot \nabla e^{\left( t- \tau \right) \Delta} f \left( u \left( \tau \right) \right) d \tau}_{L^{1} \cap L^{\infty}} =0,
	\end{align*}
	which implies
	\begin{align*}
		\limsup_{t \searrow 0} Q \left( t \right) \leq C_{2} \norm{u_{0}}_{L^{1}_{1} \cap L^{\infty}}.
	\end{align*}
	Therefore, if
	\begin{align}
		\label{eq:initial_small}
		C_{2} \norm{u_{0}}_{L^{1}_{1} \cap L^{\infty}} +C_{3} \left( 2C_{2} \norm{u_{0}}_{L^{1}_{1} \cap L^{\infty}} \right)^{p} <2C_{2} \norm{u_{0}}_{L^{1}_{1} \cap L^{\infty}},
	\end{align}
	then we have
	\begin{align*}
		Q \left( t \right) \leq 2C_{2} \norm{u_{0}}_{L^{1}_{1} \cap L^{\infty}}
	\end{align*}
	for all $t>0$.
	Combining this estimate with H\"{o}lder's inequality, we arrive at \eqref{eq:P_decay_above_2} with $C_{1} \coloneqq 2C_{2}$.
	The condition \eqref{eq:initial_small} can be rewritten as
	\begin{align*}
		\norm{u_{0}}_{L^{1}_{1} \cap L^{\infty}} < \frac{1}{2C_{2}} \left( 2C_{3} \right)^{- \frac{1}{p-1}} \eqqcolon \varepsilon_{0}.
	\end{align*}
	This completes the proof of Theorem \ref{thm:P_decay_above} (2).
\end{proof}

\begin{rem}
	In the proofs of Lemma \ref{lem:P_Duhamel_decay} and Theorem \ref{thm:P_decay_above} (2), the condition $p \leq 1+1/n$ is not used, and therefore the argument remains valid even in the case where $p>1+1/n$.
\end{rem}

\section{Proofs of Theorems \ref{thm:P_asymp} and \ref{thm:P_decay_below}} \label{sec:thm_2-3}

We first give the proof of Theorem \ref{thm:P_asymp}.

\begin{proof}[Proof of Theorem \ref{thm:P_asymp}]
	Since $p>1+1/ ( n+1 )$, it follows from \eqref{eq:psi_0} and \eqref{eq:P_decay_assump} that
	\begin{align*}
		\norm{\psi_{0}}_{L^{1}}
		\leq \int_{0}^{\infty} \norm{u \left( \tau \right)}_{L^{p}}^{p} d \tau
		\leq C \int_{0}^{\infty} \left( 1+ \tau \right)^{- \frac{n}{2} \left( p-1 \right) - \frac{p}{2}} d \tau
		\leq C,
	\end{align*}
	which implies $\psi_{0} \in L^{1} \left( \R^{n} \right)$.

	We show that
	\begin{align}
		\label{eq:P_Duhamel_asymp}
		\lim_{t \to \infty} t^{\frac{n}{2} \left( 1- \frac{1}{q} \right) + \frac{1}{2}} \norm{\int_{0}^{t} a \cdot \nabla e^{\left( t- \tau \right) \Delta} f \left( u \left( \tau \right) \right) d \tau -a \cdot \nabla e^{t \Delta} \psi_{0}}_{L^{q}} =0
	\end{align}
	holds for any $q \in \left[ 1, \infty \right]$.
	Let $q \in \left[ 1, \infty \right]$ and let $t>1$.
	We decompose the remainder as follows:
	\begin{align*}
		\int_{0}^{t} a \cdot \nabla e^{\left( t- \tau \right) \Delta} f \left( u \left( \tau \right) \right) d \tau -a \cdot \nabla e^{t \Delta} \psi_{0}
		&= \int_{0}^{t/2} \int_{0}^{1} a \cdot \nabla \left( - \tau \Delta \right) e^{\left( t- \tau \theta \right) \Delta} f \left( u \left( \tau \right) \right) d \theta d \tau \\
		&\hspace{1cm} + \int_{t/2}^{t} a \cdot \nabla e^{\left( t- \tau \right) \Delta} f \left( u \left( \tau \right) \right) d \tau \\
		&\hspace{1cm} -a \cdot \nabla e^{t \Delta} \int_{t/2}^{\infty} f \left( u \left( \tau \right) \right) d \tau \\
		&\eqqcolon \sum_{\ell =1}^{3} J_{\ell} \left( t \right).
	\end{align*}
	By virtue of \eqref{eq:P_decay_assump} and Lemma \ref{lem:heat_decay}, each term is estimated as
	\begin{align*}
		\norm{J_{1} \left( t \right)}_{L^{q}}
		&\leq \int_{0}^{t/2} \int_{0}^{1} \norm{a \cdot \nabla \left( - \tau \Delta \right) e^{\left( t- \tau \theta \right) \Delta} f \left( u \left( \tau \right) \right)}_{L^{q}} d \theta d \tau \\
		&\leq C \int_{0}^{t/2} \int_{0}^{1} \tau \left( t- \tau \theta \right)^{- \frac{n}{2} \left( 1- \frac{1}{q} \right) - \frac{3}{2}} \norm{u \left( \tau \right)}_{L^{p}}^{p} d \theta d \tau \\
		&\leq Ct^{- \frac{n}{2} \left( 1- \frac{1}{q} \right) - \frac{3}{2}} \int_{0}^{t/2} \left( 1+ \tau \right)^{1- \frac{n}{2} \left( p-1 \right) - \frac{p}{2}} d \tau \\
		&\leq Ct^{- \frac{n}{2} \left( 1- \frac{1}{q} \right) - \frac{3}{2}} \times \begin{dcases}
			\left( 1+t \right)^{- \frac{n+1}{2} \left( p-1 \right) + \frac{3}{2}} , &\qquad p<1+ \frac{3}{n+1}, \\
			\log \left( 2+t \right), &\qquad p=1+ \frac{3}{n+1}, \\
			1, &\qquad p>1+ \frac{3}{n+1}
		\end{dcases} \\
		&\leq Ct^{- \frac{n}{2} \left( 1- \frac{1}{q} \right) - \frac{1}{2}} \times \begin{dcases}
			t^{- \frac{n+1}{2} \left( p-1 \right) + \frac{1}{2}} , &\qquad p<1+ \frac{3}{n+1}, \\
			t^{-1} \log \left( 2+t \right), &\qquad p=1+ \frac{3}{n+1}, \\
			t^{-1}, &\qquad p>1+ \frac{3}{n+1},
		\end{dcases} \\
		\norm{J_{2} \left( t \right)}_{L^{q}}
		&\leq \int_{t/2}^{t} \norm{a \cdot \nabla e^{\left( t- \tau \right) \Delta} f \left( u \left( \tau \right) \right)}_{L^{q}} d \tau \\
		&\leq C \int_{t/2}^{t} \left( t- \tau \right)^{- \frac{1}{2}} \norm{u \left( \tau \right)}_{L^{pq}}^{p} d \tau \\
		&\leq C \int_{t/2}^{t} \left( t- \tau \right)^{- \frac{1}{2}} \left( 1+ \tau \right)^{- \frac{n}{2} \left( 1- \frac{1}{q} \right) - \frac{n}{2} \left( p-1 \right) - \frac{p}{2}} d \tau \\
		&\leq Ct^{- \frac{n}{2} \left( 1- \frac{1}{q} \right) - \frac{n}{2} \left( p-1 \right) - \frac{p}{2}} \int_{t/2}^{t} \left( t- \tau \right)^{- \frac{1}{2}} d \tau \\
		&\leq Ct^{- \frac{n}{2} \left( 1- \frac{1}{q} \right) - \frac{n+1}{2} \left( p-1 \right)}, \\
		\norm{J_{3} \left( t \right)}_{L^{q}}
		&= \norm{a \cdot \nabla e^{t \Delta} \int_{t/2}^{\infty} f \left( u \left( \tau \right) \right) d \tau}_{L^{q}} \\
		&\leq Ct^{- \frac{n}{2} \left( 1- \frac{1}{q} \right) - \frac{1}{2}} \int_{t/2}^{\infty} \norm{u \left( \tau \right)}_{L^{p}}^{p} d \tau \\
		&\leq Ct^{- \frac{n}{2} \left( 1- \frac{1}{q} \right) - \frac{1}{2}} \int_{t/2}^{\infty} \left( 1+ \tau \right)^{- \frac{n}{2} \left( p-1 \right) - \frac{p}{2}} d \tau \\
		&\leq Ct^{- \frac{n}{2} \left( 1- \frac{1}{q} \right) - \frac{n+1}{2} \left( p-1 \right)}.
	\end{align*}
	Taking into account the fact that
	\begin{align*}
		- \frac{n+1}{2} \left( p-1 \right) <- \frac{1}{2} \quad \Longleftrightarrow \quad p>1+ \frac{1}{n+1},
	\end{align*}
	we obtain \eqref{eq:P_Duhamel_asymp}.

	Since $\cal{M}_{0} \left[ u_{0} \right] =0$, by Lemma \ref{lem:heat_asymp} and \eqref{eq:P_Duhamel_asymp}, we have
	\begin{align*}
		&\limsup_{t \to \infty} t^{\frac{n}{2} \left( 1- \frac{1}{q} \right) + \frac{1}{2}} \norm{u \left( t \right) - \scr{A} \left( t \right)}_{L^{q}} \\
		&\hspace{1cm} \leq \limsup_{t \to \infty} t^{\frac{n}{2} \left( 1- \frac{1}{q} \right) + \frac{1}{2}} \norm{e^{t \Delta} u_{0} - \cal{M}_{0} \left[ u_{0} \right] G_{t} + \cal{M}_{1} \left[ u_{0} \right] \cdot \nabla G_{t}}_{L^{q}} \\
		&\hspace{2cm} + \limsup_{t \to \infty} t^{\frac{n}{2} \left( 1- \frac{1}{q} \right) + \frac{1}{2}} \norm{\int_{0}^{t} a \cdot \nabla e^{\left( t- \tau \right) \Delta} f \left( u \left( \tau \right) \right) d \tau -a \cdot \nabla e^{t \Delta} \psi_{0}}_{L^{q}} \\
		&\hspace{2cm} + \limsup_{t \to \infty} t^{\frac{n}{2} \left( 1- \frac{1}{q} \right) + \frac{1}{2}} \norm{a \cdot \nabla \left( e^{t \Delta} \psi_{0} - \cal{M}_{0} \left[ \psi_{0} \right] G_{t} \right)}_{L^{q}} \\
		&\hspace{1cm} =0,
	\end{align*}
	which yields \eqref{eq:P_asymp}.

	A simple computation shows \eqref{eq:P_asymp_self-similar}.
	In order to prove the assertion (2) in Theorem \ref{thm:P_asymp}, let $q \in \left[ 1, \infty \right]$ and assume that $\scr{B} \in C \left( \left( 0, \infty \right); L^{q} \left( \R^{n} \right) \right)$ satisfies $\scr{B} \left( t \right) =t^{- \frac{1}{2}} \delta_{t} \scr{B} \left( 1 \right)$ for all $t>0$ and
	\begin{align*}
		\lim_{t \to \infty} t^{\frac{n}{2} \left( 1- \frac{1}{q} \right) + \frac{1}{2}} \norm{u \left( t \right) - \scr{B} \left( t \right)}_{L^{q}} =0.
	\end{align*}
	Then we have
	\begin{align*}
		t^{\frac{n}{2} \left( 1- \frac{1}{q} \right) + \frac{1}{2}} \norm{\scr{A} \left( t \right) - \scr{B} \left( t \right)}_{L^{q}} = \norm{\scr{A} \left( 1 \right) - \scr{B} \left( 1 \right)}_{L^{q}}
	\end{align*}
	for any $t>0$.
	Moreover, we obtain
	\begin{align*}
		&\limsup_{t \to \infty} t^{\frac{n}{2} \left( 1- \frac{1}{q} \right) + \frac{1}{2}} \norm{\scr{A} \left( t \right) - \scr{B} \left( t \right)}_{L^{q}} \\
		&\hspace{1cm} \leq \limsup_{t \to \infty} t^{\frac{n}{2} \left( 1- \frac{1}{q} \right) + \frac{1}{2}} \norm{u \left( t \right) - \scr{A} \left( t \right)}_{L^{q}} + \limsup_{t \to \infty} t^{\frac{n}{2} \left( 1- \frac{1}{q} \right) + \frac{1}{2}} \norm{u \left( t \right) - \scr{B} \left( t \right)}_{L^{q}} \\
		&\hspace{1cm} =0.
	\end{align*}
	Therefore, we conclude that $\norm{\scr{A} \left( 1 \right) - \scr{B} \left( 1 \right)}_{L^{q}} =0$, which implies
	\begin{align*}
		\scr{A} \left( t \right) =t^{- \frac{1}{2}} \delta_{t} \scr{A} \left( 1 \right) =t^{- \frac{1}{2}} \delta_{t} \scr{B} \left( 1 \right) = \scr{B} \left( t \right)
	\end{align*}
	for all $t>0$.
\end{proof}

Theorem \ref{thm:P_decay_below} follows as a corollary of Theorem \ref{thm:P_asymp}.

\begin{proof}[Proof of Theorem \ref{thm:P_decay_below}]
	Let $q \in \left[ 1, \infty \right]$.
	It follows from \eqref{eq:P_asymp_self-similar} that
	\begin{align*}
		\norm{\scr{A} \left( t \right)}_{L^{q}} =t^{- \frac{n}{2} \left( 1- \frac{1}{q} \right) - \frac{1}{2}} \norm{\scr{A} \left( 1 \right)}_{L^{q}}
	\end{align*}
	for all $t>0$.
	By \eqref{eq:P_asymp}, we have
	\begin{align*}
		\limsup_{t \to \infty} t^{\frac{n}{2} \left( 1- \frac{1}{q} \right) + \frac{1}{2}} \norm{u \left( t \right)}_{L^{q}}
		&\leq \limsup_{t \to \infty} t^{\frac{n}{2} \left( 1- \frac{1}{q} \right) + \frac{1}{2}} \norm{\scr{A} \left( t \right)}_{L^{q}} \\
		&\hspace{1cm} + \limsup_{t \to \infty} t^{\frac{n}{2} \left( 1- \frac{1}{q} \right) + \frac{1}{2}} \norm{u \left( t \right) - \scr{A} \left( t \right)}_{L^{q}} \\
		&= \norm{\scr{A} \left( 1 \right)}_{L^{q}}, \\
		\liminf_{t \to \infty} t^{\frac{n}{2} \left( 1- \frac{1}{q} \right) + \frac{1}{2}} \norm{u \left( t \right)}_{L^{q}}
		&\geq \liminf_{t \to \infty} t^{\frac{n}{2} \left( 1- \frac{1}{q} \right) + \frac{1}{2}} \norm{\scr{A} \left( t \right)}_{L^{q}} \\
		&\hspace{1cm} - \limsup_{t \to \infty} t^{\frac{n}{2} \left( 1- \frac{1}{q} \right) + \frac{1}{2}} \norm{u \left( t \right) - \scr{A} \left( t \right)}_{L^{q}} \\
		&= \norm{\scr{A} \left( 1 \right)}_{L^{q}},
	\end{align*}
	which implies \eqref{eq:P_decay_below}.
	We note that $\scr{A} \left( 1 \right)$ is represented as
	\begin{align*}
		\scr{A} \left( 1 \right) = \frac{1}{2} \sum_{j=1}^{n} \left( \cal{M}_{e_{j}} \left[ u_{0} \right] -a_{j} \cal{M}_{0} \left[ \psi_{0} \right] \right) x_{j} G_{1},
	\end{align*}
	where $a_{j}$, $x_{j}$, and $\cal{M}_{e_{j}} \left[ u_{0} \right]$ denote the $j$th components of $a$, $x$, and $\cal{M}_{1} \left[ u_{0} \right]$, respectively.
	Since
	\begin{align*}
		\frac{1}{2} \int_{\R^{n}} y_{j} y_{k} G_{1} \left( y \right) dy= \begin{cases}
			1, &\qquad j=k, \\
			0, &\qquad j \neq k,
		\end{cases}
	\end{align*}
	we conclude that $\norm{\scr{A} \left( 1 \right)}_{L^{q}} >0$ if and only if $\cal{M}_{e_{j}} \left[ u_{0} \right] -a_{j} \cal{M}_{0} \left[ \psi_{0} \right] \neq 0$ for some $j \in \left\{ 1, \ldots, n \right\}$, namely, $\cal{M}_{1} \left[ u_{0} \right] -a \cal{M}_{0} \left[ \psi_{0} \right] \neq 0$.
\end{proof}

\appendix
\section{A sufficient condition for (\ref{eq:P_decay_moment})} \label{sec:app_moment}

In this section, we provide a sufficient condition on the initial datum under which \eqref{eq:P_decay_moment} holds.

\begin{prop} \label{pro:appendix}
	Let $p>1+1/ ( n+1 )$ and let $\varphi \in ( L^{1}_{1} \cap L^{\infty} ) \left( \R^{n} \right) \setminus \left\{ 0 \right\}$ satisfy $\cal{M}_{0} \left[ \varphi \right] =0$ and $\cal{M}_{1} \left[ \varphi \right] \neq 0$.
	Let $u_{0} \coloneqq \varepsilon \varphi$ with $\varepsilon >0$ and let $u \in X$ be the global solution to \eqref{P} given in Proposition \ref{pro:P_global}.
	Then there exists $\varepsilon_{1} >0$, independent of $\varepsilon$, such that \eqref{eq:P_decay_moment} holds, provided that $\varepsilon < \varepsilon_{1}$.
\end{prop}

\begin{proof}
	We first recall that, as mentioned in Remark \ref{rem:smallness}, Theorem \ref{thm:P_decay_above} (2) remains valid even in the case where $p>1+1/n$.
	Assume $\varepsilon < \varepsilon_{0} / \norm{\varphi}_{L^{1}_{1} \cap L^{\infty}}$, where $\varepsilon_{0} >0$ is the constant given in Theorem \ref{thm:P_decay_above}.
	Then it follows from \eqref{eq:psi_0}, Theorem \ref{thm:P_decay_above}, and the assumption $p>1+1/ ( n+1 )$ that
	\begin{align*}
		\abs{\cal{M}_{0} \left[ \psi_{0} \right]}
		&\leq \int_{0}^{\infty} \norm{u \left( \tau \right)}_{L^{p}}^{p} d \tau \\
		&\leq C_{1}^{p} \norm{u_{0}}_{L^{1}_{1} \cap L^{\infty}}^{p} \int_{0}^{\infty} \left( 1+ \tau \right)^{- \frac{n}{2} \left( p-1 \right) - \frac{p}{2}} d \tau \\
		&= \frac{2C_{1}^{p} \norm{\varphi}_{L^{1}_{1} \cap L^{\infty}}^{p}}{\left( n+1 \right) \left( p-1 \right) -1} \varepsilon^{p}.
	\end{align*}
	By using this estimate, we have
	\begin{align*}
		\abs{\cal{M}_{1} \left[ u_{0} \right] -a \cal{M}_{0} \left[ \psi_{0} \right]}
		&\geq \abs{\cal{M}_{1} \left[ u_{0} \right]} - \abs{a} \abs{\cal{M}_{0} \left[ \psi_{0} \right]} \\
		&\geq \varepsilon \abs{\cal{M}_{1} \left[ \varphi \right]} - \frac{2 \abs{a} C_{1}^{p} \norm{\varphi}_{L^{1}_{1} \cap L^{\infty}}^{p}}{\left( n+1 \right) \left( p-1 \right) -1} \varepsilon^{p} \\
		&= \varepsilon \left( \abs{\cal{M}_{1} \left[ \varphi \right]} - \frac{2 \abs{a} C_{1}^{p} \norm{\varphi}_{L^{1}_{1} \cap L^{\infty}}^{p}}{\left( n+1 \right) \left( p-1 \right) -1} \varepsilon^{p-1} \right).
	\end{align*}
	Since $\cal{M}_{1} \left[ \varphi \right] \neq 0$, by setting
	\begin{align*}
		\varepsilon_{1} \coloneqq \min \left\{ \frac{\varepsilon_{0}}{\norm{\varphi}_{L^{1}_{1} \cap L^{\infty}}}, {\,} \left( \frac{\left( n+1 \right) \left( p-1 \right) -1}{4 \abs{a} C_{1}^{p} \norm{\varphi}_{L^{1}_{1} \cap L^{\infty}}^{p}} \abs{\cal{M}_{1} \left[ \varphi \right]} \right)^{\frac{1}{p-1}} \right\},
	\end{align*}
	we obtain
	\begin{align*}
		\abs{\cal{M}_{1} \left[ u_{0} \right] -a \cal{M}_{0} \left[ \psi_{0} \right]} > \frac{\varepsilon}{2} \abs{\cal{M}_{1} \left[ \varphi \right]} >0,
	\end{align*}
	provided that $\varepsilon < \varepsilon_{1}$.
\end{proof}

\section*{Acknowledgments}

The second author was supported by JSPS Grant-in-Aid for JSPS Fellows Grant Number JP26KJ0016.
The authors would like to thank the anonymous referees for their helpful suggestions and comments.



\bigskip
\noindent
Yi C. Huang \\
School of Mathematical Sciences, \\
Nanjing Normal University, Nanjing 210023, China \\
E-mail: \texttt{yi.huang.analysis@gmail.com}

\bigskip
\noindent
Ryunosuke Kusaba \\
Graduate School of Science, \\
Tohoku University, 6-3, Aramaki Aza-Aoba, Aoba-ku, Sendai 980-8578, Japan \\
E-mail: \texttt{ryunosuke.kusaba.a1@tohoku.ac.jp}

\end{document}